\documentclass{llncs}
\usepackage{amssymb}
\usepackage{amsmath,fancyhdr}
\usepackage{amsfonts}
\usepackage{epsfig}
\usepackage{enumerate}
\usepackage{verbatim}
\usepackage{graphics}
\usepackage{hyperref}
\usepackage{makeidx}
\usepackage{subfig}
\usepackage{microtype}

\def\qed{ \ \vrule width.2cm height.2cm depth0cm\smallskip}

\newcommand{\rephrase}[3]{\noindent\textbf{#1 #2}.~\emph{#3}}

\newif\iflong

\longfalse

\begin{document}

\title{Thrackles: An Improved Upper Bound}

\author{Radoslav Fulek\inst{1}\thanks{The author greatfully acknowledges  support from Austrian Science Fund (FWF): M2281-N35.} and J\'anos Pach\inst{2}\thanks{Supported by Swiss National Science Foundation Grants 200021-165977 and 200020-162884.}}

\institute{
IST Austria, Am Campus 1, Klosterneuburg 3400, Austria\\
\email{radoslav.fulek@gmail.com}
\and
\'Ecole Polytechnique F\'ed\'erale de Lausanne, Station 8, Lausanne 1015, Switzerland and R\'enyi Institute, Hungarian Academy of Sciences, 
P.O.Box 127 Budapest, 1364, Hungary \\
\email{ pach@cims.nyu.edu}
}

\maketitle

\begin{abstract}
A {\em thrackle} is a graph drawn in the plane so that every pair of its edges meet exactly once: either at a common end vertex or in a proper crossing.
We prove that any thrackle of $n$ vertices has at most $1.3984n$ edges.
{\em Quasi-thrackles} are defined similarly, except that every pair of edges that do not share a vertex are allowed to cross an {\em odd} number of times. It is also shown that the maximum number of edges of a quasi-thrackle on $n$ vertices is ${3\over 2}(n-1)$, and that this bound is best possible for infinitely many values of $n$.
\end{abstract}

\section{Introduction}

Conway's thrackle conjecture~\cite{W71} is one of the oldest open problems in the theory of topological graphs. A {\em topological graph} is a graph drawn in the plane so that its vertices are represented by points and its edges by continuous arcs connecting the corresponding points so that (i) no arc passes through any point representing a vertex other than its endpoints, (ii) any two arcs meet in finitely many points, and (iii) no two arcs are tangent to each other. A {\em thrackle} is a topological graph in which any pair of edges (arcs) meet precisely once. According to Conway's conjecture, every thrackle of $n$ vertices can have at most $n$ edges. This is analogous to Fisher's inequality~\cite{F40}: If every pair of edges of a hypergraph $H$ have precisely one point in common, then the number of edges of $H$ cannot exceed the number of vertices.

The first linear upper bound on the number of edges of a thrackle, in terms of the number of vertices $n$, was established in~\cite{LPS97}. This bound was subsequently improved in~\cite{CN00} and~\cite{FP11}, with the present record, $1.4n$, held by Goddyn and Xu~\cite{GX17}, which also appeared in the master thesis of the second author~\cite{Xu14}. One of the aims of this note is to show that this latter bound is not best possible.

\begin{theorem}
\label{thm:thrackle}
Any thrackle on $n>3$ vertices has at most $1.3984n$ edges.
\end{theorem}

Several variants of the thrackle conjecture have been considered. For example, Ruiz-Vargas, Suk, and T\'oth~\cite{RST16} established a linear upper bound on the number of edges even if two edges are allowed to be {\em tangent} to each other. The notion of {\em generalized thrackles} was introduced in~\cite{LPS97}: they are topological graphs in which any pair of edges intersect an {\em odd} number of times, where each point of intersection is either a common endpoint or a proper crossing. A generalized thrackle in which no two edges incident to the same vertex have any other point in common is called a {\em quasi-thrackle}. We prove the following.

\begin{theorem}
\label{lemma:quasiUpper}
Any quasi-thrackle on $n$ vertices has at most
$\frac{3}{2}(n-1)$ edges, and this bound is tight for infinitely many values of $n$.
\end{theorem}

The proof of Theorem~\ref{thm:thrackle} is based on a refinement of parity arguments developed by Lov\'asz {\em et al.}~\cite{LPS97}, by Cairns--Nikolayevsky~\cite{CN00}, and by Goddyn--Xu ~\cite{GX17}, and it heavily uses the fact that two adjacent edges cannot have any other point in common. Therefore, one may suspect, as the authors of the present note did, that Theorem~\ref{thm:thrackle} generalizes to quasi-thrackles. Theorem~\ref{lemma:quasiUpper} refutes this conjecture.

\section{Terminology}
\label{sec:prelim}

Given a topological graph $G$ in the projective or Euclidean plane, if it leads to no confusion, we will make no distinction in notation or terminology between its vertices and edges and the points and arcs representing them.
A topological graph with no crossing is called an {\em embedding}. A connected component of the complement of the union of the vertices and edges of an embedding is called a {\em face}. A \emph{facial walk} of a face is a closed walk in $G$ obtained by traversing a component of the boundary of $F$. (The boundary of $F$ may consist of several components.) The same edge can be traversed by a walk at most twice; the \emph{length} of the walk is the number of edges counted with multiplicities. The edges of a walk form its \emph{support}.

A pair of faces, $F_1$ and $F_2$, in an embedding are \emph{adjacent} (or {\em neighboring}) if there exists at least one edge traversed by a facial walk of $F_1$ and a facial walk of $F_2$.
In a connected graph, the \emph{size} of a face is the length of its (uniquely determined) facial walk.
A face of size $k$ (resp., at least $k$ or at most $k$) is called a \emph{$k$-face} (resp., $k^+$-face and $k^-$-face).

A \emph{cycle} of a graph $G$ is a closed walk along edges of $G$ without vertex repetition. (To emphasize this property, sometimes we talk about ``simple'' cycles.)
A cycle of length $k$ is called a \emph{$k$-cycle}.

A simple closed curve on a surface is said to be \emph{one-sided} if its removal does not disconnect the surface.
Otherwise, it is \emph{two-sided}. An embedding of a graph $G$ in the projective plane is called a \emph{parity embedding} if every odd cycle of $G$ is one-sided and every even cycle of $G$ is two-sided.
In particular, in a parity embedding every face is of even size.

\section{Proof of Theorem~\ref{thm:thrackle}}

\label{sec:thrackle}

For convenience, we combine two theorems from \cite{CN09} and \cite{LPS97}.

\begin{corollary}
\label{thm:parityEmbedding}
A graph $G$ is a generalized thrackle if and only if  $G$ admits a parity embedding in the projective plane.
In particular, any bipartite thrackle can be embedded in the (Euclidean) plane.
\end{corollary}

\begin{proof}
If $G$ is a non-bipartite generalized thrackle, then, by a result of Cairns and Nikolayevsky~\cite[Theorem 2]{CN09}, it admits a  parity embedding in the projective plane.

On the other hand, Lov\'asz, Pach, and Szegedy~\cite[Theorem 1.4]{LPS97} showed that a bipartite graph is a generalized thrackle if and only if it is planar, in which case it
can be embedded in the projective plane so that every cycle is two-sided.\qed
\end{proof}

The proof of the next lemma is fairly simple and is omitted in this version.

\begin{lemma}
\label{lemma:3-cycles}
A thrackle does not contain more than one triangle.
\end{lemma}

Next, we prove Theorem~\ref{thm:thrackle} for triangle-free graphs. Our proof uses a refinement of the discharging method of Goddyn and Xu~\cite{GX17}.

\begin{lemma}
\label{lemma:3-cycleFree}
Any triangle-free thrackle on $n>3$ vertices has at most $1.3984(n-1)$ edges.
\end{lemma}

\begin{proof}
Since no 4-cycle can be drawn as a thrackle, the lemma holds for graphs with fewer than 5 vertices.
We claim that a vertex-minimal counterexample to the lemma is (vertex) 2-connected.
Indeed, let $G=G_1\cup G_2$, where $|V(G_1) \cap V(G_2)|<2\le |V(G_1)|,|V(G_2)|$. Suppose that $|V(G_1)|=n'$.
By the choice of $G$, we have
$|E(G)|=  |E(G_1)| + |E(G_2)| \le  1.3984(n'-1)+1.3984(n-n') = 1.3984(n - 1)$.

Thus, we can assume that $G$ is 2-connected.
Using Corollary~\ref{thm:parityEmbedding}, we can embed $G$ as follows.
If $G$ is not bipartite, we construct a  parity embedding of $G$
in the projective plane.
If $G$ is bipartite, we construct an embedding of $G$ in the Euclidean plane.
Note that in both cases, the size of each face of the embedding is even.

The following statement can be verified by a simple case analysis. It was removed from the short version of this note.

\begin{proposition}
\label{claim:cycles}
In the parity embedding of a 2-connected thrackle in the projective plane, the facial walk of every $8^{-}$-face is a cycle, that is, it has no repeated vertex.
\end{proposition}

To complete the proof of Lemma~\ref{lemma:3-cycleFree}, we use a discharging argument.
Since $G$ is embedded in the projective plane, by Euler's formula
we have
\begin{equation}
\label{eqn:Euler}
 e+1\le n+f
\end{equation}
where $f$ is the number of faces and $e$ is the number of edges of the embedding.

We put a charge $d(F)$ on each face $F$ of $G$, where $d(F)$ denotes the size of $F$, that is, the length of its facial walk.
An edge is called \emph{bad} if it is incident to a 6-face.
Let $F$ be an $8^+$-face.
Through every bad edge $uv$ of $F$, we discharge from its charge a  charge of $1/6$ to the neighboring 6-face on the other side of $uv$.

We claim that every face ends up with a charge  at least $7$.
Indeed, we proved in~\cite{FP11} that in a thrackle no pair of 6-cycles can share a vertex.
By Proposition~\ref{claim:cycles}, $G$ has no 8-face with 7 bad edges. Furthermore, every $8^{-}$-face is a 6-face or an 8-face, since in a parity embedding there is no odd face, and 4-cycles are not thrackleable.

\begin{proposition}
\label{claim:8^+}
Unless $G$ has 12 vertices and 14 edges, no two $8^+$-faces that share an edge can end up with charge precisely 7.
\end{proposition}

\begin{proof}
An 8-face $F$ with charge 7 must be adjacent to a pair of 6-faces, $F_1$ and $F_2$.
By Proposition~\ref{claim:cycles}, the facial walks of $F,F_1,$ and $F_2$ are cycles. Since $G$ does not contain a cycle of length 4, both $F_1$ and $F_2$ share three edges with $F$, or one of them shares two edges with $F$ and the other one four edges. Hence, any 8-face $F'$ adjacent to $F$ shares an edge $uv$ with
$F$, whose both endpoints are incident to a 6-face.
If $F'$ has charge 7, both edges adjacent to $uv$ along the facial walk $F'$ must be incident to a 6-face. By the aforementioned result from~\cite{FP11},
these 6-faces must be $F_1$ and $F_2$. By Proposition~\ref{claim:cycles}, the facial walk of $F'$ is an 8-cycle.
Since $F'$ shares 6 edges with $F_1$ and $F_2$, we obtain that $G$
has only 4 faces $F,F',F_1,$ and $F_2$.\qed
\end{proof}

In the case where $G$ has 12 vertices and 14 edges, the lemma is true.
By Proposition~\ref{claim:8^+}, if a pair of $8^+$-faces share an edge, at least one of them ends up with a charge at least $43/6$. Let $F$ be such a face. We can further discharge $1/24$ from the charge of $F$ to each neighboring $8^+$-face. After this step, the remaining charge of $F$ is at least ${43\over 6}-3{1\over 24}=7+{1\over 24}$, which is possibly attained only by an 8-face that shares 5 edges with 6-faces. Every $9^+$-face $F'$ has charge at least $d(F')-{d(F')\over 6}\ge 7+{1\over 2}$.


%

In the last discharging step, we discharge through each bad edge of an $8^+$-face an additional charge of $1/288$ to the neighboring $6$-face. At the end, the charge of every face is at least $7+{1\over 24}-6{1\over 288}=7+{1\over 48}$. Since the total charge $\sum_F d(F)=2e$ has not changed during the procedure, we obtain
$2e\ge (7+{1\over 48})f$. Combining this with~(\ref{eqn:Euler}), we conclude that $$ e\le \frac{7+{1\over 48}}{5+{1\over 48}}n-\frac{7+{1\over 48}}{5+{1\over 48}}\le 1.3984(n-1),$$
which completes the proof of Lemma~\ref{lemma:3-cycleFree}.\qed
\end{proof}

Now we are in a position to prove Theorem~\ref{thm:thrackle}.

\medskip

\noindent{\it Proof of Theorem~\ref{thm:thrackle}. }
 If $G$ does not contain a triangle, we are done by Lemma~\ref{lemma:3-cycleFree}.
Otherwise, $G$ contains a triangle $T$. We remove an edge of $T$ from $G$
and denote the resulting graph by $G'$.
According to Lemma~\ref{lemma:3-cycles}, $G'$ is triangle-free.
Hence, by Lemma~\ref{lemma:3-cycleFree}, $G'$ has at most $1.3984(n-1)$ edges, and
it follows that $G$ has at most $1.3984(n-1)+1<1.3984n$ edges.\qed

\medskip

\begin{figure}[htp]

\centering
\includegraphics[scale=0.5]{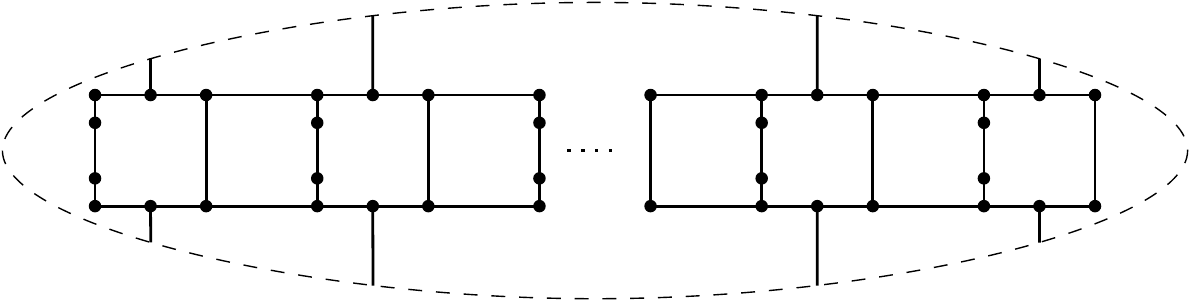}
\caption{Graph $H(k)$ embedded in the projective plane such that the embedding is a parity embedding. The projective plane is obtained by identifying the opposite pairs of points on the ellipse.
}
\label{fig:86}
\end{figure}

\begin{remark}
Without introducing any additional forbidden configuration, our methods cannot lead to an upper bound in  Theorem~\ref{thm:thrackle}, better than ${22\over 16}n=1.375n$.
\end{remark}

This is a simple consequence of the next lemma.
Let $H(k)$ be a graph obtained by taking the union of a pair of vertex-disjoint paths $P=p_1\ldots p_{6k}$ and $Q=q_1\ldots q_{6k}$
 of length $6k$; edges $p_iq_i$ for all $i\mod 3= 0 $; edges $p_iq_{6k-i}$ for all $i\mod 3 = 2$; and paths $p_ip_i'p_i''q_i$, for all $i\mod 3=1$, which are internally
 vertex-disjoint from $P, Q,$ and from one another.

 \begin{lemma}
\label{lemma:86}
For every $k\in \mathbb{N}$, the graph $H(k)$ has $16k$ vertices and $22k-2$ edges, it contains no two 6-cycles that share a vertex or are joined by an edge, and  it
admits a parity embedding in the projective plane.
\end{lemma}
\begin{proof}
For every $k$, $H(k)$ has $12k-4$ vertices of degree three and $4k+4$ vertices of degree two.
Thus, $H(k)$ has $3(6k-2)+4k+4=22k-2$ edges.
A projective embedding of $G(k)$ with the required property is depicted in Figure~\ref{fig:86}.
Using the fact that all 6-cycles are facial, the lemma follows.\qed
\end{proof}

\begin{remark}
It was stated without proof in ~\cite{CN09} that the thrackle conjecture has been verified by computer up to $n=11$. Provided that this is true, the upper bound in Theorem~\ref{thm:thrackle} can be improved to
$ e\le \frac{7+{1\over 5}}{5+{1\over 5}}(n-1)\le 1.3847(n-1)$. This follows from the fact that in this case an $8$-face and a $6$-face can share at most
one edge and therefore we can maintain a charge of at least $7+\frac 15$ on every face.
\end{remark}

\section{Proof of Theorem~\ref{lemma:quasiUpper}}
\label{sec:quasiUpper}

It is known~\cite{CN00} that $C_4$, a cycle of length 4, can be drawn as a generalized thrackle.
Hence, our next result whose simple proof is left to the reader implies that the class of quasi-thrackles forms a proper subclass of
the class of generalized thrackles.

\begin{lemma}
\label{lemma:c4}
$C_4$ cannot be drawn as a quasi-thrackle.
\end{lemma}

Let $G(k)$ denote a graph consisting of $k$ pairwise edge-disjoint triangles that intersect in a single vertex.
The drawing of $G(3)$ as a quasi-thrackle, depicted in Figure~\ref{fig:32n}, can be easily generalized to any $k$. Therefore, we obtain the following
\begin{lemma}
\label{lemma:quasiLower}
For every $k$, the graph $G(k)$ can be drawn as a quasi-thrackle.
\end{lemma}

In view of Lemma~\ref{lemma:3-cycles}, $G_k$ cannot be drawn as a thrackle for any $k>1$.
Thus, the class of thrackles is
a proper sub-class of the class of quasi-thrackles.

\begin{figure}[htp]
\centering
\subfloat[]{
\includegraphics[scale=0.5]{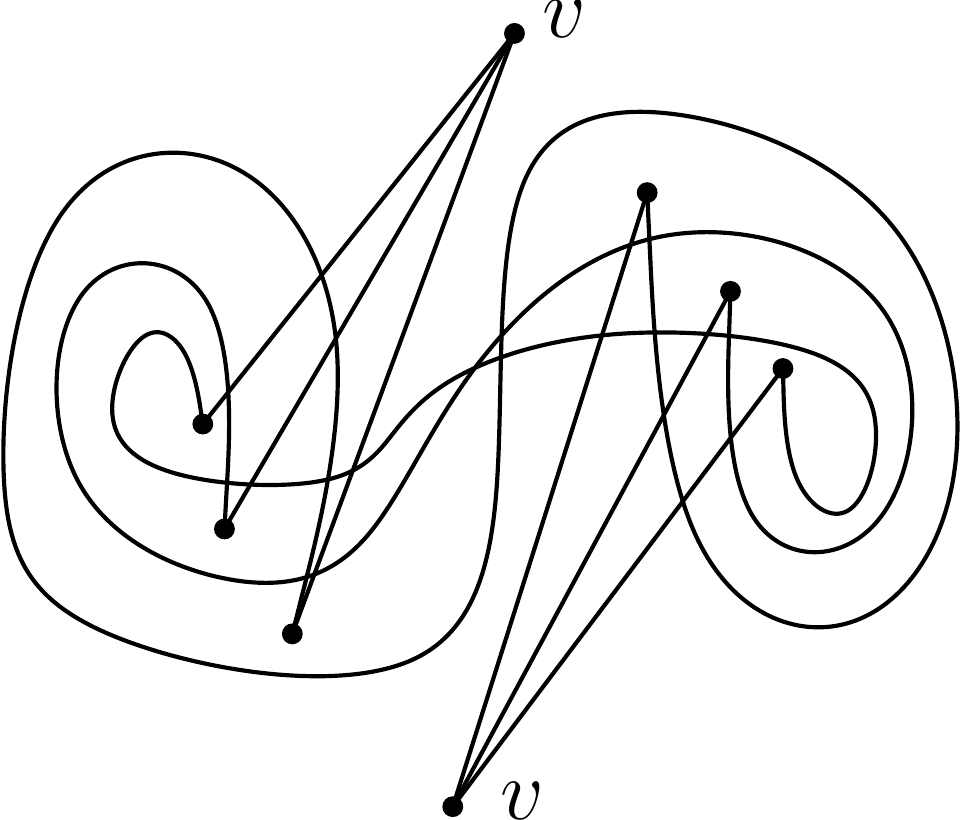}\label{fig:32n}}
\caption{A drawing of $G(3)$ as a quasi-thrackle. The two copies of the vertex $v$ are identified in the actual drawing.}
\end{figure}

Cairns and Nikolayevsky~\cite{CN00} proved that every generalized thrackle of $n$ vertices has at most $2n-2$ edges, and that this bound cannot be improved.
The graphs $G(k)$ show that for $n=2k+1$, there exists a quasi-thrackle with $n$ vertices and with $\frac{3}{2}(n-1)$ edges. According to Theorem~\ref{lemma:quasiUpper}, no quasi-thrackle with $n$ vertices can have more edges.

\medskip

\noindent{\it Proof of Theorem~\ref{lemma:quasiUpper}. }
 Suppose that the theorem is false, and let $G$ be a counterexample with the minimum number $n$ of vertices.

We can assume that $G$ is 2-vertex-connected.
Indeed, otherwise $G=G_1\cup G_2$, where $|V(G_1) \cap V(G_2)|\ge 1$
and $E(G_1) \cap E(G_2)=\emptyset$. Suppose that $|V(G_1)|=n'$.
By the choice of $G$, we have
$|E(G)|=  |E(G_1)| + |E(G_2)| \le  \frac32(n'-1)+\frac32(n-n') = \frac32(n - 1)$, so $G$ is not a counterexample.

Suppose first that $G$ is bipartite. By Corollary~\ref{thm:parityEmbedding}, $G$ (as an abstract graph) can be embedded in the Euclidean plane.
By Lemma~\ref{lemma:c4}, all faces in this embedding are of size at least 6. Using a standard double-counting argument, we obtain that $2e\ge 6f$, where $e$ and $f$ are the number of edges and faces of $G$, respectively. By Euler's formula, we have $e+2=n+f$. Hence, $6e+12\leq 6n+2e$, and rearranging the terms we obtain $e\le{3\over 2}(n-6)$, contradicting our assumption that $G$ was not a counterexample.

If $G$ is not bipartite, then, according to Corollary~\ref{thm:parityEmbedding}, it has a parity embedding in the
projective plane. By Lemma~\ref{lemma:c4}, $G$ contains no 4-cycle. It does not have loops and multiple edges, therefore, the embedding has no 4-face. $G$ cannot have a 5-face, because the facial walk of a 5-face would be either a one-sided 5-cycle (which is impossible), or it would contain a triangle and a cut-vertex (contradicting the 2-connectivity of $G$). The embedding of $G$ also does not have a 3-face, since $G$ is bipartite. By Euler's formula, $e+1=n+f$ and, as in the previous paragraph, we conclude that $6e+6\leq 6n+2e$, the desired contradiction.\qed

\bibliographystyle{plain}
\bibliography{bib}

\begin{thebibliography}{1}

\bibitem{CN00}
Grant Cairns and Yury Nikolayevsky.
\newblock Bounds for generalized thrackles.
\newblock {\em Discrete Comput. Geom.}, 23(2):191--206, 2000.

\bibitem{CN09}
Grant Cairns and Yury Nikolayevsky.
\newblock Generalized thrackle drawings of non-bipartite graphs.
\newblock {\em Discrete \& Computational Geometry}, 41(1):119--134, 2009.

\bibitem{F40}
Ronald~Aylmer Fisher.
\newblock An examination of the different possible solutions of a problem in
  incomplete blocks.
\newblock {\em Annals of Human Genetics}, 10(1):52--75, 1940.

\bibitem{FP11}
Radoslav Fulek and J\'anos Pach.
\newblock A computational approach to {C}onwayʼs thrackle conjecture.
\newblock {\em Computational Geometry}, 44(6–7):345 -- 355, 2011.

\bibitem{GX17}
Luis Goddyn and Yian Xu.
\newblock On the bounds on {C}onway's thrackles.
\newblock {\em Discrete \& Computational Geometry}, to appear, 2017.

\bibitem{LPS97}
L\'aszl\'o Lov{\'a}sz, J\'{a}nos Pach, and Mario Szegedy.
\newblock On {C}onway's thrackle conjecture.
\newblock {\em Discrete \& Computational Geometry}, 18(4):369--376, 1997.

\bibitem{RST16}
Andres~J. Ruiz-Vargas, Andrew Suk, and Csaba~D. T\'oth.
\newblock Disjoint edges in topological graphs and the tangled-thrackle
  conjecture.
\newblock {\em European J. Combin.}, 51:398--406, 2016.

\bibitem{W71}
Douglas~R Woodall.
\newblock Thrackles and deadlock.
\newblock {\em Combinatorial Mathematics and Its Applications}, 348:335--348,
  1971.

\bibitem{Xu14}
Yian Xu.
\newblock Generalized thrackles and graph embeddings, 2014.
\newblock M.Sc. Thesis, Simon Fraser University.

\end{thebibliography}
\newpage

\section{Omitted proofs}

\begin{figure}[htp]
\centering
\includegraphics[scale=0.5]{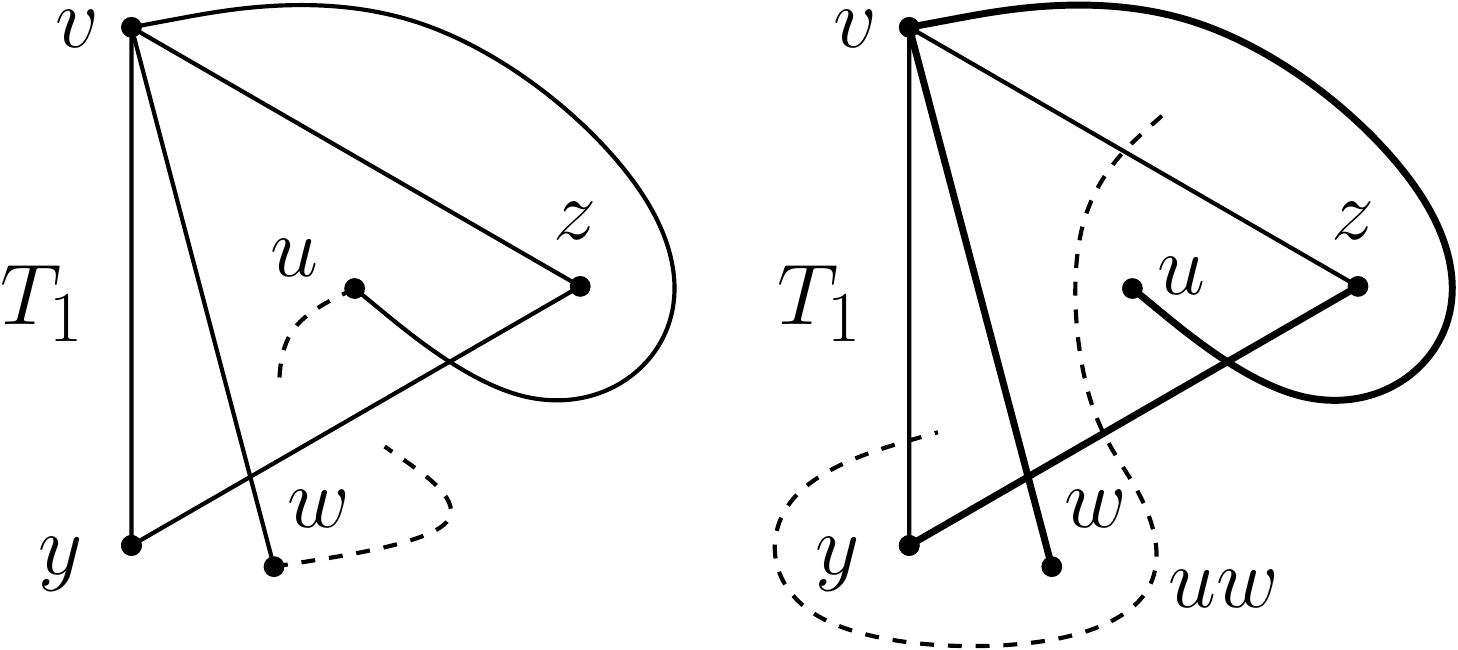}
\caption{An illustration for the proof of Lemma~\ref{lemma:3-cycles}.}
\label{fig:3-cycle}
\end{figure}

\rephrase{Lemma}{\ref{lemma:3-cycles}}{
A thrackle does not contain more than one triangle.
}
\begin{proof}
Refer to Fig.~\ref{fig:3-cycle}.
By Lemma~\cite[Lemma 2.1]{LPS97}, every pair of triangles in a thrackle share a vertex.
A pair of triangles cannot share an edge, otherwise they would form a 4-cycle, and a thrackle cannot contain a 4-cycle, since 4-cycle is not a thrackle, which is easy to check.

Let $T_1=vzy$ and $T_2=vwu$ be two triangles that have a vertex $v$ in common.
By Lemma~\cite[Lemma 2.2]{LPS97}, the two closed curves representing $T_1$ and $T_2$ properly cross each other at $v$.
Hence, the closed Jordan curve $C_1$ corresponding to $T_1$ contains $w$ in its interior and $u$ in its exterior.
Thus, the drawing of $T_1\cup \{uv,uw\}$ in a thrackle is uniquely determined up to isotopy and the choice of the outer face.
If we traverse the edge $wu$ from one endpoint to the other, we encounter its crossings  with the edges $vy,yz$, and $zv$ in this or in the reversed order.
 Indeed, the crossings between $wu$ and $vz$, and $wu$ and $vy$ must be in different
connected components of the complement of the union of $zy,vw$, and $vu$ in the plane, see Fig.~\ref{fig:3-cycle} right.
By symmetry, the crossing of $zy$ and $wu$ is on both $zy$ and $wu$ between the other two crossings.
However, this is impossible in a thrackle. We obtain a contradiction, which proves the lemma.\qed
\end{proof}

\rephrase{Proposition}{\ref{claim:cycles}}{
In the parity embedding of a 2-connected thrackle in the projective plane, the facial walk of every $8^{-}$-face is a cycle, that is, it has no repeated vertex.
}
\begin{proof}
If $G$ is bipartite, the claim follows by the 2-connectivity of $G$ and by the fact that the 4-cycle is not a thrackle.

Suppose $G$ is not bipartite. Then $G$ cannot contain $4^{-}$-face,
since we excluded triangles (by the hypothesis of the lemma) and 4-cycles (which are not thrackles).
We can also exclude any 5-face $F$, because either the facial walk of $F$ is a 5-cycle, which is impossible in a parity embedding, or the facial walk contains a triangle.

Analogously, if $F$ is a 7-face, its facial walk cannot be a  cycle (with no repeated vertex).
Hence, the support of $F$ must contain a 5-cycle. Using the fact that $G$ has no
triangle and 4-cycle, we conclude that $F$ must be incident to a cut-vertex, a contradiction.

It remains to deal with 6-faces and 8-faces.
If the facial walk of a 6-face $F$ is not a 6-cycle, then its support is a path of length three or a 3-star. In this case, $G$ is a tree on three vertices, contradicting our assumption that $G$ is 2-connected. Thus, the facial walk of every 6-face must be a 6-cycle.

The support of the facial walk of an 8-face $F$ cannot contain a 5-cycle, because in this case it would also contain a triangle. Therefore, the support of $F$ must contain a 6-cycle. The remaining (2-sided) edge of $F$ cannot be a diagonal of this cycle (as then it would create a triangle or a 4-cycle), and it cannot be a ``hanging'' edge (because this would contradict the 2-connectivity of $G$). This completes the proof of the proposition.\qed
\end{proof}

%

\rephrase{Lemma}{\ref{lemma:c4}}{
$C_4$ cannot be drawn as a quasi-thrackle.
}

\begin{proof}
Suppose for contradiction that $C_4=uvwz$ can be drawn as a quasi-thrackle; see Fig.3.
Assume without loss of generality that in the corresponding drawing, the path $uvw$ a path $uvw$ does not intersect itself.
Let $c_1$ denote the first crossing along $uz$ (with $vw$) on the way from $u$.
Let $c_2$ denote the first crossing along $wz$ (with $uv$) on the way from $w$.
Let $C_u$ denote the closed Jordan curve consisting of $uv$; the portion of
$uz$ between $u$ and $c_1$; and the portion of $vw$ between $v$ and $c_1$.
Let $C_w$ denote the closed Jordan curve consisting of $vw$; the portion of
$wz$ between $w$ and $c_2$; and the portion of $uv$ between $v$ and $c_2$.

\begin{figure}[htp]
\centering
\label{fig:c5}
\includegraphics[scale=0.5]{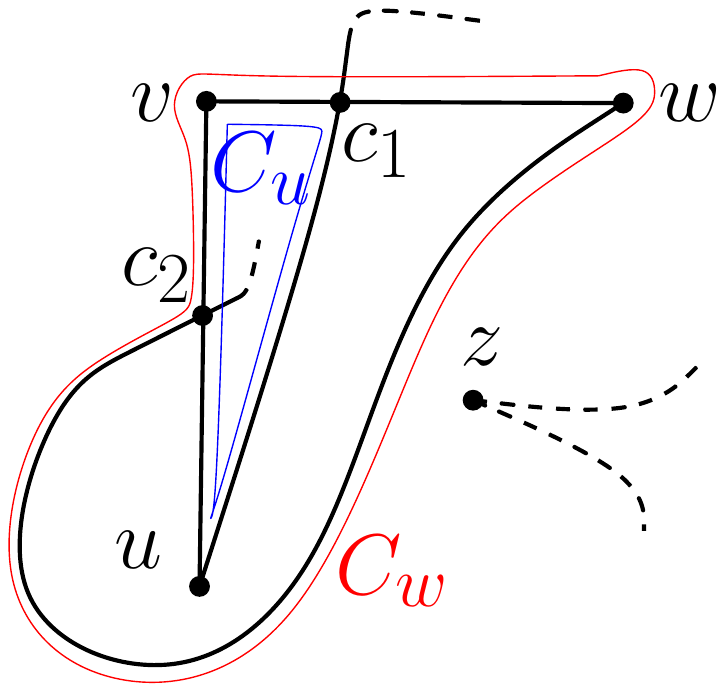}
\caption{An illustration for the proof of Lemma~\ref{lemma:c4}.}
\end{figure}

Observe that $z$ and $w$ are not contained in the same connected component of the complement
of $C_u$ in the plane. Indeed, $wz$ crosses $C_u$ an odd number of times, since it can cross
it only in $uv$. Let $\mathcal{D}_u$ denote the connected component of the complement
of $C_u$ containing $z$. By a similar argument,
$z$ and $u$ are not contained in the same connected component
of the complement of $C_w$ in the plane. Let $\mathcal{D}_w$ denote the connected component of the complement
of $C_w$  containing $z$.

Since $z\in\mathcal{D}_u \cap \mathcal{D}_w$, we have that $\mathcal{D}_u \cap \mathcal{D}_w \not=\emptyset.$
On the other hand, $C_u$ and $C_w$ do not cross each other, but they share a Jordan arc containing neither $u$ nor $w$.
If  $\mathcal{D}_u \subset \mathcal{D}_w$ (or $\mathcal{D}_w \subset \mathcal{D}_u$),
then $u$ and $z$ are both in $\mathcal{D}_w$ (or $w$ and $z$ are both in $\mathcal{D}_u$), which is impossible.
Otherwise, $u$ and $z$ are both in $\mathcal{D}_w$, and at the same time $w$ and $z$ are both in $\mathcal{D}_u$, which is again a contradiction.\qed
\end{proof}

\end{document}